\documentclass[a4paper]{amsart}
\usepackage[latin1]{inputenc}
\usepackage{latexsym}% para usar Box
\input amssym.def
\input amssym.tex
\usepackage{color}
\usepackage[all]{xy} %para diagramas
\newtheorem{definition}{Definition}[section]

\newtheorem{theorem}[definition]{Theorem}
\newtheorem{proposition}[definition]{Proposition}
\newtheorem{corollary}[definition]{Corollary}

\newtheorem{example}[definition]{Example}

\def\emph#1{{\bfseries\itshape{#1}}}
 1   %Caracteres g"ticos.
 1  %Caracteres "doble palo".

\def\R{\mathbb{R}}               %Numeros reales
\newcount\ancho \newcount\anchom \newcount\anchoa \newcount\anchob \newcount\altura
\newcommand{\ltilde}[3][0]{\altura=0 \advance\altura by #1 \ancho=#2 \anchom=\ancho \divide\anchom by 2
\anchoa=\ancho \divide\anchoa by 4 \anchob=\anchom \advance\anchob by \anchoa \kern-3pt \begin{array}[b]{c} \begin{picture}(1,1)(\anchom,-\altura)
\qbezier(0,2)(\anchoa,5)(\anchom,2) \qbezier(\anchom,2)(\anchob,-1)(\ancho,4) \qbezier(0,2)(\anchoa,4.5)(\anchom,1.8) \qbezier(\anchom,1.8)(\anchob,-1.5)(\ancho,4)
\end{picture} \\[-4pt]{#3} \end{array} \kern-4pt	}

\newcommand{\w}{\omega}

\newcommand{\lc}{T^*\phi}
\newcommand{\piq}{\pi_Q}

\newcommand{\pr}{\mathrm{pr}}
\newcommand{\Sect}{\mathrm{Sect}}

\begin{document}

\title[On symplectic lifts of actions for complete Lagrangian fibrations]{On  sympletic lifts of  actions for complete Lagrangian fibrations}

\author{Juan Carlos Marrero}
\address{J.\ C.\ Marrero:
Unidad asociada ULL-CSIC, Geometr{\'\i}a Diferencial y Mec\'anica Geom\'etrica, Departamento de Matem\'atica, Estad\'istica e I.O., Facultad de
Matem\'aticas, Universidad de la Laguna, La Laguna, Tenerife,
Canary Islands, Spain} \email{jcmarrer@ull.edu.es}

\author{Edith Padr{\'o}n}
\address{E.\ Padr{\'o}n: Unidad asociada ULL-CSIC, Geometr{\'\i}a Diferencial y Mec\'anica Geom\'etrica,Departamento de Matem\'atica, Estad\'istica e I.O., Facultad de
Matem\'aticas, Universidad de la Laguna, La Laguna, Tenerife,
Canary Islands, Spain} \email{mepadron@ull.edu.es}

\noindent

\thanks{\noindent The authors have been
partially supported by the Ministerio de Ciencia e Innovaci\'on (Spain) project  MTM2015-64166-C2-1-P. We also would like to thank  D. Rodr{\'\i}guez and M. Teixid\'o  for their useful comments. }
\noindent

\begin{abstract}
In this note we discuss symplectic lifts of actions for a complete Lagrangian fibration. Firstly, we describe the symplectic cotangent lifts of a $G$-action on a manifold $Q$ in terms of $1$-cocycles in the cohomology of $G$ induced by the action with values in the space of closed $1$-forms on $Q$. After this,  we consider the general case of complete Lagrangian fibrations. 
\end{abstract}

\maketitle

\hfill{ \it Dedicated to our friend Alberto Ibort on the occasion of his  60th birthday}

\setcounter{section}{0}

\section{Introduction}

A Lagrangian fibration $\pi$ is a surjective submersion with total space a symplectic manifold $M$ and such that the fibers of $\pi$ are Lagrangian submanifolds of $M.$ It is well-known that Lagrangian fibrations are closely related with the theory of Arnold-Liouville of completely integrable systems \cite{AM,A} and this is a good motivation for the discussion of such objects.  The typical example of a Lagrangian fibration is the canonical projection from the cotangent bundle $T^*Q$ of a manifold $Q$ over $Q$. This Lagrangian fibration is complete because the vertical lift of every $1$-form on $Q$ is a complete vector field on $T^*Q$ (see, for instance, \cite{LR} for the definition of the vertical lift of a $1$-form to the cotangent bundle). 

\medskip

The vertical lift to the total space of an arbitrary Lagrangian fibration $\pi$ of a $1$-form on the base space is also well-defined vector field and, by analogy with the case of the cotangent bundle, $\pi$ is said to be complete if such vector fields are complete. 

\medskip

Every complete Lagrangian fibration $\pi:M\to Q$ defines a new complete Lagrangian fibration with  total space the quotient $T^*Q/\Lambda$ and base space $Q$, where $\Lambda=\bigcup_{q\in Q}\Lambda_q$ is a Lagrangian submanifold of $T^*Q$ and $\Lambda_q$ is a discrete subgroup of the additive group of $T_q^*Q.$ In fact, $T^*_qQ/\Lambda_q$ is (non-canonically) diffeomorphic to the Lagrangian fiber $\pi^{-1}(q).$  However,  $M$ and $T^*Q/\Lambda$ are not, in general, globally isomorphic. Under a strong hyphothesis, the presence of a global Lagrangian section of the fibration $\pi:M\to Q,$ one may prove that there exists a global fiber-preserving symplectomorphism between $M$ and $T^*Q/\Lambda$. The Lagrangian fibration $\widetilde\pi_Q:T^*Q/\Lambda\to Q$ is called the symplectic reference of $\pi:M\to Q$ (for more details, see \cite{D}).

\medskip

On the other hand,  it is well-known that the standard  example of a sympletic action on the cotangent bundle of a manifold $Q$ is the cotangent lift of an action $\phi$ on $Q.$ This action is very interesting from a mathematical and physical point of view. In fact, it plays an important role in the study of symmetric Hamiltonian systems (see, for instance, \cite{AM, MMOPR, MR, OR}).
\medskip

The aim of this note is to discuss symplectic lifts of actions for complete Lagrangian fibrations. First of all, we consider the particular case when the Lagrangian fibration is the standard projection $\pi_Q:T^*Q\to Q.$ In such a case, we prove that every symplectic lift of a $G$-action $\phi$ on $Q$ is the composition of the cotangent lift of $\phi$ with a translation. This translation is given by a $1$-cocycle in the cohomology of $G$ induced by $\phi$ with values in the 
space of closed $1$-forms on $Q$. Moreover, the symplectic action is completely determined, up to isomorphism, by the cohomology class of the $1$-cocycle. In the general case of an arbitrary complete Lagrangian fibration $\pi:M\to Q,$ we prove similar results using the symplectic reference $\widetilde\pi_Q:T^*Q/ \Lambda\to Q$ of $\pi:M\to Q.$ 

\medskip

The note is structured as follows. In Section \ref{Section2}, we discuss the symplectic lifts of actions for the standard Lagrangian fibration $\pi_Q:T^*Q\to Q$ and,  in Section \ref{Section3},  we obtain the corresponding results for an arbitrary complete Lagrangian fibration.

\section{symplectic cotangent lifts of actions on a manifold}\label{Section2}

Let $Q$ be a manifold of dimension $n.$  We denote by $\theta_Q$ the Liouville $1$-form, by $\w_Q=-d\theta_Q$ the canonical symplectic structure  on the cotangent bundle $T^*Q$ and by $\pi_Q:T^*Q\to Q$ the corresponding projection of $T^*Q$ on $Q$.

\medskip 

In addition, we suppose that we have a  left action $\phi:G\times Q\to Q$ of a Lie group $G$ on $Q$. Denote by  $T^*\phi: G\times T^*Q \to T^*Q$  the cotangent lift action  given by
$$ (T^*\phi)_g\alpha_q:=T^*_{\phi_g(q)}\phi_{g^{-1}}(\alpha_q),\;\;\;\; \mbox{ for }g\in G \mbox{ and } \alpha_q\in T^*_qQ.$$

The cotangent lift $T^*\phi$  of the action $\phi:G\times Q\to Q$  is a  fiberwise symplectic action, i.e.
$$\pi_Q\circ (T^*\phi)_g=\phi\circ \pi_Q,\mbox{ and }(T^*\phi)_g^*\w_Q=\w_Q,\;\;\;\mbox{ for all } g\in G.$$

\medskip

 Now, we ask about if it is possible to consider other symplectic actions on $(T^*Q,\omega_Q)$ which fiber on $\phi$. 

\medskip 

It is well-known that a diffeomorphism $F:T^*Q\to T^*Q$ on $T^*Q$ preserves the Liouville $1$-form $\theta_Q$  if and only if there is a diffeomorphim $f:Q\to Q$ such that $F=T^*f$ (see Proposition 6.3.2 in \cite{MR}).  So, if $F$ is a fiber-preserving diffeomorphism then $F$ is the identity. 

Moreover, we may prove the following result. 

\begin{proposition}\label{lema}
Let $F: T^*Q \to T^*Q$ be a diffeomorphism which is fibered over the identity. The following statements  are equivalent:
\begin{enumerate}
\item[$(i)$]  $F$ is symplectic with respect to the canonical  symplectic structure  $\w_Q$ on $T^*Q$. 
\item[$(ii)$] There exists a unique  closed 1-form $\alpha$ on $Q$ such that $F$ is just the fiber translation $t_\alpha:T^*Q\to T^*Q$ by $\alpha$, that is, 
$$t_\alpha(\gamma_q)=\gamma_q+\alpha(q), \; \mbox{ for all } \gamma_q \in T_q^*Q.$$
\end{enumerate}
\end{proposition}

\begin{proof}
Suppose that $F$ is the fiber translation $t_\alpha$ by a 1-form $\alpha$ on $Q$. Then, a direct proof using local coordinates allows to obtain the following formula
\begin{equation}\label{t_alfa}
F^*\w_Q=t_\alpha^*\w_Q=\w_Q-\pi^*_Q(d\alpha).
\end{equation}
Thus, if $\alpha$ is closed, 
$$
F^*\w_Q =\w_Q. 
$$

Conversely, if $(i)$ is satisfied, we  can consider a $1$-form $\alpha \in \Omega^1(Q)$ given by
$$\alpha(q)=F(\gamma_q)-\gamma_q, $$
where $\gamma_q$ is an arbitrary element of $T^*_qQ$. We will see that, for each $q\in Q,$ this definition does not depend on the choice of $\gamma_q \in T^*_qQ$, or equivalently, the linear map $T_{\gamma_q}(F-id_{T^*Q})$ is null on vertical vectors. Note that, since $F$ is fibered over the identity, the map $F-id_{T^*Q}$ is well-defined. 

\medskip

Let $\lambda$ be a 1-form on $Q$, then its vertical lift $X_{\pi_Q^*\lambda}$ is a vector field on $T^*Q$  which is characterized in terms of the canonical sympletic structure on $T^*Q$ as follows (see \cite{LR})
\begin{equation}\label{lambda^v}
i_{X_{\pi_Q^*\lambda}}\w_Q=\pi_Q^*\lambda.
\end{equation}

In fact, we have that (see \cite{LR}) 
\begin{equation}\label{def-v}
X_{\pi_Q^*\lambda}(\beta_q)=\frac{d}{dt}_{|t=0}(\beta_q + t\lambda(q)) \mbox{ for }\beta_q\in T^*_qQ.
\end{equation}
This implies that the vertical bundle $V_{\beta_q}\pi_Q$ of $\pi_Q$ at the point $\beta_q\in T^*_qQ$ is 
$$V_{\beta_q}\pi_Q=\{X_{\pi^*_Q\lambda}(\beta_q)/\lambda\in \Omega^1(Q)\}.$$

Moreover, if $f$ is a fiberwise linear function on $T^*Q$ then there exists a vector field $Y$ on $Q$ such that $f=\widehat{Y}$, that is, 
\begin{equation}\label{Y}
f(\beta_q)=<\beta_q,Y(q)>, \mbox{ for } \beta_q\in T^*_qQ,
\end{equation}
and one may prove that 
\begin{equation}\label{def-vv}
X_{\pi_Q^*\lambda}(\widehat{Y})=\lambda(Y)\circ \pi_Q.
\end{equation}

Now, from (\ref{lambda^v}) and using that $F$ is symplectic, we have that 
\begin{eqnarray*}
i_{X_{\pi_Q^*\lambda}}F^*(\w_Q)=i_{X_{\pi_Q^*\lambda}}\omega_Q=\pi_Q^*\lambda =(\pi_Q\circ F)^*\lambda=F^*(\pi^*_Q\lambda)=F^*(i_{X_{\pi_Q^*\lambda}}\omega_Q).
\end{eqnarray*}

As a consequence,  using the non-degeneracy of $\w_Q$, we deduce that 
\begin{equation}\label{TF}
T_{\gamma_q}F(X_{\pi_Q^*\lambda}(\gamma_q))=X_{\pi_Q^*\lambda}(F(\gamma_q))\mbox{ for all } \gamma_q\in T^*_qQ. 
\end{equation} 

\medskip

On the other hand, since $\pi_Q\circ F=\pi_Q$ then  $T_{\gamma_q}(F-id_{T^*Q})(X_{\pi_Q^*\lambda}(\gamma_q))$ is a vertical vector, so it is characterized by its value on fiberwise linear functions $f \in C^\infty(T^*Q)$. Thus, using (\ref{def-vv}) and (\ref{TF}), we have that
$$
T_{\gamma_q}(F-id_{T^*Q})(X_{\pi_Q^*\lambda}({\gamma_q}))(f)= X_{\pi_Q^*\lambda}({F(\gamma_q)})(f)-X_{\pi_Q^*\lambda}({\gamma_q})(f)=0, 
$$
which proves the proposition. 

\end{proof} 

Now, suppose that we have an arbitrary  symplectic action $\Phi:G\times T^*Q\to T^*Q$   such that projects on an action $\phi:Q\to Q$, i.e.  the following diagram commutes
\[
\xymatrix{T^*Q \ar[d]_{\pi_Q} \ar[rr]^{\Phi_g} &&T^*Q \ar[d]^{\pi_Q} \\
Q \ar[rr]^{\phi_g} &&Q
}
\]

and  $\Phi_g^*(\w_Q)=\w_Q,$  for all $g\in G.$ We ask about  how are these  symplectic actions on $T^*Q$. The following result gives an answer to this question.

\begin{proposition}\label{T2}
Let $\Phi:G\times T^*Q\to T^*Q $ be a symplectic  action of a Lie group $G$ on  $(T^*Q,\w_Q)$ whose projection on $Q$ is the action $\phi:G\times Q\to Q.$  Then, 
there exists a differentiable map  $A:G \times Q \to T^*Q$ such that 
\begin{enumerate}
\item $A$ is fibered on $Q$, i.e. $\pi_Q\circ A=pr_2,$ where $pr_2:G\times Q\to Q$ is the canonical projection on the second factor. 
\item For each $g\in G$, the $1$-form $A_g$ on $Q$ is closed. 
\item The action $\Phi$ is given by 
\begin{equation}\label{PhiA}
\Phi_g=(\lc)_g \circ t_{A(g)}, \mbox{ for all } g \in G.
\end{equation}
\end{enumerate} 
\medskip

In particular,  $\Phi$ is an affine action. Moreover,  $\Phi$ is linear if and only if  $\Phi$ is the cotangent lift $T^*\phi$ of $\phi$.  
\end{proposition}
\begin{proof}
Let $g$ be an element of the Lie group $G$. The map $F_g=(\lc)_{g^{-1}} \circ \Phi_g$ satisfies the hypothesis of  Proposition \ref{lema}. Thus,  there exists a unique closed 1-form $A_g$ on $Q$ such that $\Phi_g=(\lc)_g \circ t_{A(g)}$.   
\end{proof}

Therefore, if $\Omega^1_c(Q)$ denotes the set of closed $1$-forms on $Q$, every symplectic fiberwise action is affine and it induces a map $A:G\to \Omega_c^1(Q)$ satisfying (\ref{PhiA}) but, does each one of these maps induces an affine symplectic fiberwise action?
 The following result give us the necessary and sufficient conditions on the map $A:G\to \Omega_c^1(Q)$ to ensure that the map $\Phi^A:G\times T^*Q\to T^*Q$ 
  related with $A$ by (\ref{PhiA}) is a symplectic action. Previously, we introduce the following cohomology complex induced by the action $\phi$ (see, for instance, \cite{AI}):
\begin{itemize}
\item A $n$-cochain is a map $A: G \times \stackrel{n}{\dots} \times G \to \Omega^1(Q)$ and $C^n(G,\Omega^1(Q))$ denotes the set of the $n$-cochains. The $0$-cochains are the $1$-forms on $Q$. 
\item The coboundary operator  $\delta_\phi: C^n(G,\Omega^1(Q)) \to C^{n+1}(G,\Omega^1(Q))$ is given by
\begin{eqnarray*}
(\delta_\phi A) (g_1, \dots, g_{n+1}) &=& (-1)^{n+1} A(g_2, \dots, g_{n+1}) + \\
&&+ \sum^n_{i=1} (-1)^{n+i+1} A(g_1, \dots, g_{i-1}, g_i\cdot g_{i+1}, \dots, g_{n+1}) +\\
&&+ \phi^*_{g_{n+1}} (A(g_1, \dots, g_n)). 
\end{eqnarray*}
\end{itemize}
Since the exterior differential is linear and commutes with the pull backthen the sets $C^n(G,\Omega^1_c(Q))$ of the  $n$-cochains with values in the closed 1-forms on $Q$ define a subcomplex of $(C^\bullet(G,\Omega^1(Q)),\delta_\phi)$. We denote by $H^k(G,\phi,\Omega^1_c(Q))$ the corresponding cohomology groups. We will see that the first cohomology group $H^1(G,\phi,\Omega_c^1(Q))$ allows to classify the symplectic  actions on $T^*Q$ which project on $\phi.$

\begin{theorem}\label{T3}
Let $\phi: G \times Q \to Q$ be an action of the Lie group $G$ on a manifold $Q$ and $A$ be a map from $G$ to $\Omega^1(Q)$. Then,  
\begin{enumerate}
\item The map $\Phi^A: G \times T^*Q \to T^*Q$ given by $\Phi_g^A=(\lc)_g \circ t_{A(g)}$ is an action if and only if $A$ is a one-cocycle in the cohomology complex $(C^\bullet(G,\Omega^1(Q)),$ $\delta_\phi)$. 
\item $\Phi^A$  is also symplectic if and only if $A(g)$ is a closed 1-form on $Q,$ for all $g \in G$, i.e. $A$ is a one-cocycle in the cohomology subcomplex $(C^\bullet(G,\Omega^1_c(Q)),\delta_\phi)$.
\end{enumerate}
\end{theorem}

\begin{proof}
The left action condition $ \Phi_{gh}^A=\Phi_g^A\circ \Phi_h^A$ is equivalent to the relation
$$ (\lc)_{gh}(A(gh)(q))=(\lc)_{gh}(A(h)(q))+(\lc)_g(A(g)({\phi_h(q)})).$$
If we apply $(\lc)_{(gh)^{-1}}$ in the previous equality, then we have
$$ A(gh)=A(h)+\phi^*_h(A(g)), \qquad \forall \; g,h \in G. $$
Note that this last condition implies that  $A(e)=0$, where $e$ is the identity element  of $G$. 
Therefore, $\Phi^A$ is an action if and only if $A$ is a one-cocycle in the cohomology complex $(C^\bullet(G,\Omega^1(Q)),\delta_\phi).$

\medskip

On the other hand, using (\ref{t_alfa}) and the symplectic character of the cotangent lift action, we deduce
\begin{eqnarray*}
(\Phi_g^A)^*(\w_Q) &=& t_{A(g)}^*((\lc)_g^*\w_Q )= t_{A(g)}^*\w_Q \\
&=& \w_Q-\piq^*(d(A(g))).
\end{eqnarray*} 
Thus, $\Phi_g^A$ is symplectic if and only if $A(g)$ is a closed 1-form on $Q,$ for all $g \in G$.
\end{proof} 

Note that the previous results give us a relation between symplectic actions on $(T^*Q,\w_Q)$  whose projection on $Q$ is $\phi:G\times Q\to Q$ and one-cocycles in the cohomology complex $(C^\bullet(G,\Omega^1_c(Q)),\delta_\phi).$  Using these facts, we will see in the following theorem that the  first cohomology  group $H^1(G,\phi, \Omega^1_c(Q))$ of this complex  allows to give a classification of these symplectic actions on $(T^*Q, \w_Q)$. 

\begin{theorem}\label{T4}
Let $\phi: G \times Q \to Q$ be an action of a  Lie group $G$ on a manifold $Q.$
If $A, B: G \to \Omega^1_c(Q)$ are  two one-cocycles in the cohomology complex $(C^\bullet(G,\Omega_c^1(Q)),\delta_\phi)$, and $\Phi^A$ and $\Phi^B$ their respective affine symplectic actions on $(T^*Q, \w_Q)$, then  exists a symplectomorphism $F:T^*Q\to T^*Q$  such that 
\begin{eqnarray}
\label{F fibrada}          && \piq \circ F= \piq \mbox{ and }\\ 
\label{F equivariante} && F \circ \Phi^A_g=\Phi^B_g \circ F, \mbox{ for all } g \in G 
\end{eqnarray}
if and only if $[A]=[B] \in H^1(G,\phi, \Omega^1_c(Q))$.  
\end{theorem}
\begin{proof}
Suppose that  $F:T^*Q\to T^*Q$ is a symplectomorphism which satisfies  \eqref{F fibrada} and \eqref{F equivariante}. Then, using  Proposition \ref{lema}, we deduce that there exists a closed 1-form $\alpha$ on $Q$ such that $F=t_\alpha$. Since $t_\alpha \circ \Phi^A_g=\Phi^B_g \circ t_\alpha$, we have that
$$(\lc)_g(\gamma_q+A(g)(q)+\alpha({\phi_g(q)}))=(\lc)_g(\gamma_q+B(g)(q)+\alpha(q)),$$
for all $\gamma_q \in T_q^*Q$. Thus, we deduce 
$$(A-B)(g)=\alpha-\phi^*_g\alpha=\delta_\phi(-\alpha)(g),$$
that is, $[A]=[B]$. 

\medskip

Conversely, if $[A]=[B]$, then there exists a closed 1-form $\alpha$ on $Q$ such that $(A-B)(g)=\delta_\phi(\alpha)(g)=\phi^*_g\alpha-\alpha$. Thus, $F=t_{-\alpha}$ satisfies the conditions of the theorem.
\end{proof}

\begin{example}[{\bf The Heisenberg group}] {\rm 
Let $Q={\Bbb R}^3\cong{\Bbb R}^2\times {\Bbb R}$ be 
the Heisenberg group  endowed with the group law 
$$(x,y,t)\cdot (x',y',t')=(x+x',y+ y', t+t'+xy'),\;\;\; \mbox{ for all }(x,y,t),(x',y',t')\in {\Bbb R}^3.$$

We consider the action  $\phi: {\Bbb R}^3\times {\Bbb R}^3\to {\Bbb R}^3$ defined by the group operation. 
The cotangent lift action $T^*\phi$ is 
$$(T^*\phi)_{(x_0,y_0,t_0)}(x,y,t,\alpha_1,\alpha_2,\alpha_3)=(x+x_0,y+y_0,t+t_0+x_0y,\alpha_1,\alpha_2-x_0\alpha_3,\alpha_3).$$
Here we have identified $T^*{\Bbb R}^3\cong {\Bbb R}^6.$  

Let $A:{\Bbb R}^3\to \Omega^1_c({\Bbb R}^3)$ be the one-cocycle 
$$A(x_0,y_0,t_0)=x_0^2dy+2x_0dt \mbox{ for all } (x_0,y_0,t_0)\in {\Bbb R}^3.$$

Note that 
$$A(x_0+x_1,y_0+y_1,t_0+t_1+x_0y_1)=A(x_1,y_1,t_1)+\phi_{(x_1,y_1,t_1)}^*A_{(x_0,y_0,t_0)}.$$

\medskip

Suppose that $A$ is a coboundary  in the  cohomogy complex $(C^\bullet({\Bbb R}^3,\Omega^1_c({\Bbb R}^3)),\delta_\phi)$. Then there exists a closed $1$-form $\alpha=\alpha_1dx + \alpha_2dy + \alpha_3 dt$ such that $A{(x_0,y_0,t_0)}=\phi_{(x_0,y_0,t_0)}^*\alpha-\alpha$ for all $(x_0,y_0,t_0)\in {\Bbb R}^3,$ i.e.
$$
\begin{array}{rcl}
0&=&\alpha_1\circ \phi_{(x_0,y_0,t_0)}-\alpha_1,\\
 x^2_0&=&\alpha_2\circ \phi_{(x_0,y_0,t_0)}+ x_0\alpha_3\circ \phi_{(x_0,y_0,t_0)}-\alpha_2,\\
 2x_0&=&\alpha_3\circ \phi_{(x_0,y_0,t_0)}-\alpha_3.
 \end{array}
 $$

At the point $(0,0,0) \in {\Bbb R}^3$ we have that 
$$
\begin{array}{rcl}
\alpha_1(x_0,y_0,t_0)&=&\alpha_1(0,0,0),\\
\alpha_2(x_0,y_0,t_0)&=&x_0^2 -x_0\alpha_3(0,0,0) +\alpha_2(0,0,0), \\
\alpha_3(x_0,y_0,t_0)&=&2x_0 + \alpha_3(0,0,0).
\end{array}$$

On the other hand, since $\alpha$ is closed 

$$\frac{\partial \alpha_1}{\partial y}- \frac{\partial\alpha_2}{\partial x}=0,$$
that is, $2x_0 -\alpha_3(0,0,0)=0$ which it is not possible. Thus,  $[A]\not=0.$

The new symplectic action on $T^*{\Bbb R}^3$ is 
$$\Phi_{(x_0,y_0,s_0)}(x,y,t,\alpha_1,\alpha_2,\alpha_3)=(x+x_0,y+y_0,t+t_0+x_0y,\alpha_1,\alpha_2-x_0^2-x_0\alpha_3
,\alpha_3+ 2x_0), $$
and $(T^*{\Bbb R}^3,\Phi)$ is not symplectomorphic to $(T^*{\Bbb R}^3,T^*\phi).$}

\end{example}

%%%%%%%%%%%%%%%%%%%%%%%%%%%%%%%%%%%%%%
\section{symplectic lifts of actions on a complete $G$-Lagrangian fibration}\label{Section3}

The cotangent projection $\pi_Q:T^*Q\to Q$ of a manifold $Q$  has a special property: for all $q\in Q$, its  fiber $T_q^*Q$ at $q$  is a Lagrangian submanifold of $T^*Q, $ that is,  $\pi_Q:T^*Q\to Q$ is a Lagrangian fibration.  In this section we will  extend the previous results to this kind of fibrations (with certain topological restrictions). In order to do this, we recall some notions  and properties about Lagrangian fibrations (for more details, see \cite{D}).

 \medskip
 
A fiber bundle $\pi: M \to Q$, with total space a symplectic manifold $(M,\w)$, is called a \textit{Lagrangian fibration} if its fiber $\pi^{-1}(q)$ is a Lagrangian submanifold of $M,$ for all $q \in Q$, that is, 
\begin{equation}\label{w}
\ker T_x\pi=(\ker T_x\pi)^\w,\;\;\;\mbox{ for all } x\in \pi^{-1}(q),
\end{equation}
where $(\ker T_x\pi)^\w=\{v\in T_xM/\w_x(v,u)=0 \mbox{ for all } u\in \ker T_x\pi\}$  is the symplectic orthogonal subspace of $\ker T_x\pi.$ 

\medskip 

Consider  $(M,Q,\pi,\w)$ a Lagrangian fibration. Given a 1-form $\alpha$ on $Q$, denote by  $X_{\pi^*\alpha}$ the vertical vector field on  $M$ which is characterized by the following condition 
\begin{equation}\label{Xpi}
i_{X_{\pi^*\alpha}}\w=\pi^*\alpha.
\end{equation}
Note that (\ref{w}) implies that $X_{\pi^*\alpha}$ is a vertical vector field with respect to $\pi$ (in fact, using (\ref{w}), we deduce that the vertical bundle to $\pi$ is generated by the vector fields $X_{\pi^*\alpha}$, with $\alpha$ a $1$-form on $Q$).
So, if $F_t^{X_{\pi^*\alpha}}:M\to M$ is  the flow of $X_{\pi^*\alpha}$ at the time $t$, then $\pi\circ F_t^{X_{\pi^*(\alpha)}}=\pi$ and one can prove that 
$$\frac{d}{d t}[(F_t^{X_{\pi^*\alpha}})^*(\w)]=(F_t^{X_{\pi^*\alpha}})^*({\mathcal L}_{X_{\pi^*\alpha}}\w)=(F_t^{X_{\pi^*\alpha}})^*(\pi^*(d\alpha))=\pi^*(d\alpha).$$
Therefore, 
\begin{equation}\label{ecuacionclave}(F^{X_{\pi^*\alpha}}_t)^*\w=\w+t\pi^*(d\alpha).\end{equation}

\medskip

We will say that the Lagrangian fibration $(M,Q,\pi,\w)$ is \textit{complete} if the  vector field $X_{\pi^*\alpha}$ is complete,  for all $\alpha\in \Omega^1(Q)$. In such a case, this vector field can be integrated up to time 1 to give the map
\begin{equation}\label{mu}
\begin{array}{rcl}
\mu: \Omega^1(Q) \times M &\to & M \\
(\alpha,x)&\mapsto & F_1^{X_{\pi^*\alpha}}(x).
\end{array}
\end{equation}
In fact, one can check that (see \cite{D}) 
\begin{itemize}
\item $\mu$ is an action: $\mu(\alpha+\beta,x)=\mu(\alpha,\mu(\beta,x)),$
\item $\mu$ is $\pi$-fibered:  $\pi(\mu(\alpha,x))=\pi(x),$
\item $\mu(\alpha,x)=\mu(\beta,x)$ if $\alpha({\pi(x)})=\beta({\pi(x)}),$
\item $\mu$ is transitive on the fibers: for all $q\in Q$,  if $x,x^\prime \in \pi^{-1}(q)$, there exists a $1$-form $\alpha \in \Omega^1(Q)$ such that $\mu(\alpha,x)=x^\prime.$ \end{itemize}

Therefore, for each $q\in Q$, the action $\mu$ induces a transitive action of the abelian group $T^*_qQ$ on the fiber $\pi^{-1}(q)$
\begin{equation}\label{muuu}
\mu_q: T^*_qQ \times \pi^{-1}(q) \to \pi^{-1}(q),\;\;\;\;\;
(\alpha_q,x) \mapsto \mu(\alpha,x),
\end{equation}
where $\alpha$ is a $1$-form on $Q$ such that its value at $q$ is just $\alpha_q$. In general,  the action $\mu_q$ is not free. For this reason we consider  the isotropy subgroup $\Lambda_q$ of $\mu_q$. More explicitly, 
\begin{equation}\label{latice}
\Lambda_q=\{\alpha_q \in T_q^*Q \; | \; \mu_q(\alpha_q,x)=x \quad \forall x \in \pi^{-1}(q)\}.
\end{equation}
It can be checked that $\Lambda_q$ is a discrete subgroup of $T_q^*Q$, $T_q^*Q/\Lambda_q$ is an abelian group and  $\Lambda=\cup_{q\in Q}\Lambda_q$ is a Lagrangian submanifold of $T^*Q$. In addition, we have the corresponding free fibered action 
$$\widehat{\mu}:T^*Q/\Lambda\times M\to M.$$

On the other hand, $\widetilde{\pi}_Q:T^*Q/\Lambda\to Q$ is a Lagrangian fibration with respect the induced  symplectic $2$-form $\widetilde{\omega}_Q$ on the reduced space  $T^*Q/\Lambda$  characterized by 
\begin{equation}\label{pr}
\pr^*\widetilde{\w}_Q=\w_Q,
\end{equation}
where $\pr:T^*Q \to T^*Q/\Lambda$ is the quotient projection. Thus, $(T^*Q/\Lambda,Q,\widetilde{\pi}_Q,\widetilde{\w}_Q)$ is a complete Lagrangian fibration. 
 Furthermore, if the fibers of $\widetilde{\pi}_Q:T^*Q/\Lambda\to Q$ are connected and compact, then they are isomorphic to the $n$-torus (for more details, see \cite{D}). 
\medskip

In the particular case of the Lagrangian fibration $(T^*Q,Q,\pi_Q,\omega_Q),$ the action $(\mu_Q)_q:T_q^*Q\times T_q^*Q\to T^*_qQ$ is the map $(\mu_Q)_q(\alpha_q,\beta_q)=\alpha_q+\beta_q$ and $\Lambda_q=0.$ 
\medskip

For the complete Lagragian fibration $(T^*Q/\Lambda,Q,\widetilde{\pi}_Q,\widetilde{\w}_Q)$ deduced from a complete Lagrangian fibration $(M,Q,\pi,\w),$ the corresponding action is 
\begin{equation}\label{+}(\widetilde{\mu}_Q)_q:T^*_qQ\times T_q^*Q/\Lambda_q\to T_q^*Q/\Lambda_q,\;\;\; (\widetilde{\mu}_Q)_q(\alpha_q,[\beta_q])=[\alpha_q + \beta_q],\end{equation}
and its isotropic subgroup  is just $\Lambda_q.$

Now, we will characterize the  symplectomorphisms on the fiber bundle $\widetilde{\pi}_Q:T^*Q/\Lambda\to Q$. Previously, we recall the notion of  {\it a  Lagrangian section }of  an arbitrary Lagrangian fibration $\pi:(M,\w)\to Q$ like a section $\sigma:Q\to M$ of $\pi$ such that $\sigma^*\w=0.$ In the particular case of the cotangent bundle $\pi_Q:(T^*Q,\w_Q)\to Q,$ a $1$-form $\alpha$ is Lagrangian if and if it is closed, since $\alpha^*\w_Q=-d\alpha.$

\begin{proposition}\label{lemmatilde}
Let  $(T^*Q/\Lambda,Q, \widetilde{\pi}_Q,\widetilde{\w}_Q)$ be the symplectic reference of a complete Lagrangian fibration $(M,Q,\pi,\omega)$. If $\widetilde{F}:T^*Q/\Lambda\to T^*Q/\Lambda$ is a diffeomorphism such that $\widetilde{\pi}_Q\circ \widetilde F=\widetilde{\pi}_Q$ then the following statements  are equivalent: 
\begin{enumerate}
\item $\widetilde{F}$ is a symplectomorphism, that is, $\widetilde{F}^*\widetilde{\w}_Q=\widetilde{\w}_Q.$ 
\item There exists a Lagrangian  section $\widetilde\sigma:Q\to T^*Q/\Lambda$ of $\widetilde\pi_Q$  such that $ \widetilde{F}=t_{\widetilde\sigma}$, where $t_{\widetilde\sigma}:T^*Q/\Lambda\to T^*Q/\Lambda$ is the map given by 
\begin{equation}\label{tras}
t_{\widetilde\sigma}([\gamma_q] )=[\gamma_q] + {\widetilde\sigma}(q),
\end{equation}
 for all $q\in Q.$ 
\end{enumerate}
\end{proposition}
\begin{proof}
Suppose that $\widetilde{F}=t_{\widetilde\sigma}$, with ${\widetilde\sigma}$ a section  of $\widetilde{\pi}_Q:T^*Q/\Lambda\to Q.$ Then, 
\begin{equation}\label{tras}
t_{{\widetilde\sigma}}^*\widetilde{\omega}_Q=\widetilde{\omega}_Q+ \widetilde\pi_Q^*({\widetilde\sigma}^*\widetilde{\omega}_Q).
\end{equation}
In fact, for all $\gamma_q\in T_q^*Q,$ there exist a neighborhood $U$ of $q$ and a $1$-form $\alpha:U\to T^*U$ on $U$ such that $\alpha(q)=\gamma_q$ and 
\begin{equation}\label{U}pr\circ \alpha ={\widetilde\sigma}_{|U},\end{equation}
where $pr:T^*U\to T^*U/\Lambda_{|T^*U}$ is the corresponding projection. 
Then,  we have that 
\begin{equation}\label{UU}
pr\circ t_{\alpha}=t_{{\widetilde\sigma}_{|U}}\circ pr.
\end{equation}

On the other hand,  using (\ref{t_alfa}) and (\ref{U}), it follows that 
$$t_{\alpha}^*\omega_Q=\omega_Q- \pi_Q^*(d\alpha)=pr^*(\widetilde\omega_Q +\widetilde\pi_Q^*(\widetilde\sigma^*\widetilde\omega_Q)).$$
From this equality, (\ref{pr}) and (\ref{UU}), we deduce that  (\ref{tras}) holds. Therefore, if  in addition,  ${\widetilde\sigma}$ is Lagrangian,  $\widetilde F=t_{\widetilde\sigma}$ is symplectic. 

\medskip

Conversely, if $\widetilde{F}:T^*Q/\Lambda\to T^*Q/\Lambda$ is a symplectomorphism,  we define the section ${\widetilde\sigma}:Q\to T^*Q/\Lambda$ of $\widetilde{\pi}_Q:T^*Q/\Lambda \to Q$ by 
$${\widetilde\sigma}(q)=\widetilde{F}([\gamma_q])-[\gamma_q],\;\;\;\; \mbox{ with } \gamma_q\in T_q^*Q \mbox{ and }q\in Q.$$ 
In order to prove that $\widetilde\sigma$ is well defined, we will follow  the proof of Proposition \ref{lema}. So, we will show that for each $q\in Q$, the section $\widetilde{\sigma}$ doesn't depend of the chosen element $[\gamma_q]\in T_q^*Q/\Lambda_q, $  or equivalently, 
$$T_{[\gamma_q]}(\widetilde{F}-id_{T^*Q/\Lambda})(v_{[\gamma_q]})=0 \mbox{ for all $v_{[\gamma_q]}\in \ker T_{[\gamma_q]}\widetilde\pi_Q$}.$$
Let $\lambda$ be a $1$-form on $Q$ and $X_{\widetilde\pi_Q^*\lambda}$ the vertical vector field on $T^*Q/\Lambda$ of $\lambda$ with respect to $\widetilde\pi_Q,$ i.e.
\begin{equation}\label{tildelambda}
i_{X_{\widetilde\pi_Q^*\lambda}}\widetilde{\w}_Q=\widetilde\pi_Q^*\lambda.
\end{equation}
Note that 
\begin{equation}\label{lambdalambda}
X_{\widetilde\pi_Q^*\lambda}([\gamma_q])=T_{\gamma_q} pr(X_{\pi_Q^*\lambda}(\gamma_q)), \mbox{ with } \gamma_q\in T^*_qQ,
\end{equation} where $X_{\pi_Q^*\lambda}(\gamma_q)$ is the vertical lift on $T^*Q$ of $\lambda$ with respect to $\pi_Q:T^*Q\to Q$ (see (\ref{lambda^v})). In fact, using (\ref{lambda^v}), (\ref{pr}), (\ref{tildelambda}) and the fact that  $\widetilde{\pi}_Q\circ pr=\pi_Q$, we have
$$i_{X_{\pi_Q^*\lambda}}(pr^*\widetilde\omega_Q)=pr^*(i_{X_{\widetilde\pi_Q^*\lambda}}\widetilde\w_Q).$$
So, from the non-degeneration of $\widetilde\w_Q,$ we deduce (\ref{lambdalambda}). It is clear that (\ref{lambdalambda}) implies that  the vertical vectors $X_{\widetilde\pi_Q^*\lambda}([\gamma_q]),$ with $\lambda\in \Omega^1(Q),$ generate the subspace $\ker T_{[\gamma_q]}\widetilde\pi_Q.$

Now, using that $\widetilde{F}$ is symplectic and the fact that $\widetilde{\pi}_Q\circ \widetilde{F}=\widetilde{\pi}_Q,$ we deduce 
\begin{eqnarray*}
i_{X_{\widetilde\pi_Q^*\lambda}}\widetilde{F}^*(\widetilde\w_Q)=i_{X_{\widetilde\pi_Q^*\lambda}}\widetilde\omega_Q=\widetilde\pi_Q^*\lambda =(\widetilde\pi_Q\circ \widetilde{F})^*\lambda=\widetilde{F}^*(\widetilde\pi^*_Q\lambda)=\widetilde{F}^*(i_{X_{\widetilde\pi_Q^*\lambda}}\widetilde\omega_Q).
\end{eqnarray*}
Then, again from the non-degeneration of $\widetilde{\w}_Q$, 
$$ T_{[\gamma_q]}\widetilde{F}(X_{\widetilde\pi_Q^*\lambda}([\gamma_q]))=X_{\widetilde\pi_Q^*\lambda}(\widetilde{F}([\gamma_q])).$$
Therefore, 
\begin{equation}\label{v}
T_{[\gamma_q]}(\widetilde{F}-id_{T^*Q/\Lambda})(X_{\widetilde\pi_Q^*\lambda}({[\gamma_q]}))(\widetilde{f})= X_{\widetilde\pi_Q^*\lambda}({\widetilde{F}([\gamma_q])})(\widetilde{f})-X_{\widetilde\pi_Q^*\lambda}([{\gamma_q]})(\widetilde{f})=0, 
\end{equation}
for all function $\widetilde{f}:T^*Q/\Lambda\to \R$ such that $f= \widetilde{f}\circ pr$ is a linear function on $T^*Q.$ Indeed, if $Y$ is the vector field on $Q$ associated with $f$ defined  in (\ref{Y}), then,  using (\ref{lambdalambda}), we deduce that 

$$X_{\widetilde\pi_Q^*\lambda}([\gamma_q])(\widetilde{f})=<\lambda(q),Y(q)>= X_{\widetilde\pi_Q^*\lambda}(\widetilde{F}([\gamma_q]))(\widetilde{f}).$$

On the other hand, since $\widetilde\pi_Q\circ \widetilde{F}=\widetilde{\pi}_Q$, then $T_{[\gamma_q]}(\widetilde{F}-Id_{T^*Q/\Lambda})(X_{\widetilde\pi_Q^*\lambda}([\gamma_q])$ is a vertical vector with respect $\widetilde\pi_Q$. In such a case, there is a vertical vector $v_{\gamma_q}\in \ker T_{\gamma_q}\pi_Q$ such that 
$$
T_{[\gamma_q]}(\widetilde{F}-Id_{T^*Q/\Lambda})(X_{\widetilde\pi_Q^*\lambda}([\gamma_q]))=T_{\gamma_q}pr(v_{\gamma_q}).$$
From (\ref{v}), we obtain that $v_{\gamma_q}(f)=0$ for all fiberwise linear function $f:T^*Q\to \R$, which implies that $v_{\gamma_q}=0$ and,  in consequence,  $T_{[\gamma_q]}(\widetilde{F}-Id_{T^*Q/\Lambda})(X_{\widetilde\pi_Q^*\lambda}([\gamma_q]))=0.$

\end{proof}

If the complete  Lagrangian fibration $(M,Q,\pi,\w)$ has a global section $\sigma: Q \to M$ then,  using the action $\mu, $ we can build the following  fiber bundle isomorphism from the canonical fibration  $\widetilde{\pi}_Q:T^*Q/\Lambda \to Q$ to the Lagrangian fibration $\pi:M\to Q$ given by 
\begin{equation}\label{modelo}
\begin{array}{rcl}
\varphi_\sigma:T^*Q/\Lambda &\to & M \\
\left[ \alpha_q \right]&\mapsto & \mu_q(\alpha_q,\sigma(q)). 
\end{array}
\end{equation}
This map is equivariant when we consider  the additive action from $T^*Q$ over $T^*Q/\Lambda$ and the action $\mu$ on $M.$  Moreover, using (\ref{ecuacionclave}),   one can prove that   (see \cite{D})
\begin{equation}\label{simplec}
\varphi^*_\sigma{\w}=\widetilde{\w}_Q + \widetilde{\pi}_Q^*(\sigma^*\w).
\end{equation}
Then, if $\sigma$ is  Lagrangian,  $\varphi_\sigma$ is a symplectomorphism between  $(T^*Q/\Lambda,Q, \widetilde{\pi}_Q, \widetilde{\w}_Q)$ and  $(M,Q,$ $\pi,\w).$

\medskip

The existence of a global Lagrangian section depends only on the triviality of Chern class of the fiber bundle $\pi: M \to B$. In fact, in \cite{D} it is proved the following result. 

\begin{theorem}
The following statements are equivalent: 
\begin{enumerate}
\item There exists a (symplectic) fiber bundle isomorphism between  $M$ and  $T^*Q/\Lambda.$ 
\item There exists a global (Lagrangian) section $\sigma: Q\to M$ of the fiber bundle $\pi:M\to Q.$ 
\item The Chern class of the fiber bundle $\pi:M\to Q$ is null (and $\sigma^*\w$ is an exact $2$-form on $Q$).
\end{enumerate}
\end{theorem} 

These results justify that the complete Lagrangian fibration $(T^*Q/\Lambda,Q,\widetilde\pi_Q,\widetilde{\w}_Q)$ is called {\it the symplectic reference} or {\it Jacobian Lagrangian fibration} associated to the complete Lagrangian fibration $(M,Q,\pi,\w).$

\medskip

Using Proposition \ref{lemmatilde} for the symplectic reference of a complete Lagrangian fibration $(M,Q,\pi, \w),$ we deduce the corresponding result

\begin{proposition}\label{lemma2}
Let $\widehat{F}:M\to M$ be diffeomorphism on a complete Lagrangian fibration  $(M,Q,\pi, \w)$ such that ${\pi}\circ \widehat{F}={\pi}$. Then, the following statements  are equivalent: 
\begin{enumerate}
\item $\widehat{F}$ is a symplectomorphism, that is, $\widehat{F}^*{\w}={\w}.$ 
\item There exists a unique Lagrangian  section ${\widetilde\sigma}:Q\to T^*Q/\Lambda$ of the corresponding symplectic reference $(T^*Q/\Lambda,Q,\widetilde\pi_Q,\widetilde{\w}_Q)$ of $(M,Q,\pi, \w)$
  such that 
\begin{equation}\label{tras2}
\widehat{F}(x)=\widehat{\mu}({\widetilde\sigma}(\pi(x)),x),\mbox{ for all }  x\in M, 
\end{equation}
where $\widehat{\mu}:T^*Q/\Lambda \times M \to M$ is the free fibered action deduced from $\mu:T^*Q\times M\to M$ given in  (\ref{muuu}). 
\end{enumerate}
\end{proposition}

\begin{proof}
Suppose that $(ii)$ holds. Then we will prove 
\begin{equation}\label{Fw}
\widehat{F}^*\w=\w+\pi^*(\widetilde\sigma^*\widetilde{\w}_Q).
\end{equation}
In such a case, since $\widetilde\sigma^*\widetilde{\w}_Q=0,$ then $\widehat F$ is symplectic. 

Let $x$ be an arbitrary point of $M$. Then,  there exists an open neighborhood $U$ of $q=\pi(x)$ and a local section $\sigma:U\to \pi^{-1}(U)$ of $\pi$ such that the map $\varphi_\sigma:T^*U/\Lambda_{|U}\to \pi^{-1}(U)$ defined as in  (\ref{modelo}) is a fiber bundle isomorphism and 

\begin{equation}\label{varphi_sigma}
\varphi_\sigma^*{\omega}= \widetilde{\omega}_Q + \widetilde{\pi}_Q^*(\sigma^*\w), 
\end{equation}
(see (\ref{simplec})). Denote by $\widetilde{F}=\varphi_{\sigma}^{-1}\circ \widehat{F}\circ \varphi_\sigma$. Then, for all $q'\in U$ and $\alpha_{q'}\in T^*_{q'}U,$

$$\begin{array}{rcl}
\widehat{\mu}(\widetilde{F}([\alpha_{q'}]),\sigma({q'}))&=&\varphi_\sigma(\widetilde{F}([\alpha_{q'}]))=
\widehat{F}(\varphi_\sigma([\alpha_{q'}]))\\[5pt]&=& \widehat{\mu}(\widetilde{\sigma}(\pi(\varphi_\sigma([\alpha_{q'}]))), \varphi_\sigma([\alpha_{q'}]))\\[5pt] &=&  \widehat{\mu}(\widetilde\sigma({q'}), \widehat\mu([\alpha_{q'}], \sigma({q'}))=
\widehat{\mu}(\widetilde\sigma(q')+ [\alpha_{q'}],\sigma(q')).
\end{array}
$$
Now, from the free character of the action $\widehat{\mu}$ we deduce that   $\widetilde{F}=t_{\widetilde\sigma_{|U}}$ and consequently,  using (\ref{tras}) and (\ref{varphi_sigma}), we have that 
$$\begin{array}{rcl}
\widehat{F}^*(\w)&=&(\varphi_\sigma^*)^{-1}(\widetilde{F}^*(\varphi_\sigma^*\w))=(\varphi_\sigma^*)^{-1}(\widetilde{F}^*(\widetilde{\w}_Q + \widetilde{\pi}_Q^*(\sigma^*\w))\\[8pt]
&=&(\varphi_\sigma^*)^{-1}(\widetilde{\w}_Q + \widetilde{\pi}_Q^*(\sigma^*\w)+ \widetilde{\pi}_Q^*(\widetilde\sigma^*\widetilde{\w}_Q))=\w + \pi^*(\widetilde\sigma^*\widetilde{\w}_Q).
\end{array}$$
Thus, (\ref{Fw}) holds. 

\medskip

Conversely, if $\widehat{F}:M\to M$ is a symplectomorphism,  with $\pi\circ \widehat{F}=\pi,$  we consider the map $S:M\to T^*Q/\Lambda$ characterized  by 
$$\widehat{F}(x)=\widehat{\mu}(S(x), x).$$
In the following, we will prove that $S$ is a map which is constant into $\pi^{-1}(q).$  In fact, if $x\in \pi^{-1}(q),$ there exist an open neighborhood $U$ of $q$ and a local section $\sigma:U\to \pi^{-1}(U)$ of $\pi$ such that the map $\varphi_\sigma:T^*U/\Lambda_{|U}\to \pi^{-1}(U)$ defined by (\ref{modelo}) is a fiber bundle isomorphism and 
$$\varphi_\sigma^*{\w}= \widetilde{\omega}_Q + \widetilde{\pi}_Q^*(\sigma^*\w).$$ Then, using this relation and the fact that  $\widehat{F}$ is symplectic,  we obtain that 
 the map $\widetilde{F}=\varphi_{\sigma}^{-1}\circ \widehat{F}\circ \varphi_\sigma$ is symplectic and $\widetilde{\pi}_Q\circ \widetilde{F}=\widetilde{\pi}_Q.$ Now, from Proposition \ref{lemmatilde}, there exists a Lagrangian section $\widetilde{\sigma}:U\to T^*U/\Lambda_{|U}$ of $\widetilde{\pi}_Q$ such that $\widetilde{F}=t_{\widetilde{\sigma}}$. In fact, $\varphi_\sigma\circ \widetilde\sigma=\sigma.$ 
 With a  direct computation, we show that 
$$
\begin{array}{rcl}
\widehat{\mu}(S(y),y)&=&\widehat{F}(y)=\varphi_{\sigma}(\widetilde{F}(\varphi_\sigma^{-1}(y)))=\widehat{\mu}(\widetilde{\sigma}(\pi(y)), \widehat\mu(\varphi^{-1}_\sigma(y),\sigma(\pi(y))))\\&=&\widehat{\mu}(\widetilde{\sigma}(\pi(y)), \varphi_\sigma(\varphi^{-1}_\sigma(y)))=
\widehat{\mu}(\widetilde{\sigma}(\pi(y)), y),\end{array}$$

for all $y\in \pi^{-1}(U)$.  
Thus, using the free character of the action $\widehat{\mu},$ we deduce that $S_{|U}=\widetilde{\sigma}\circ \pi$ on  $U.$ 

Since $S$ is globally  defined, this proves that the Lagrangian section $\widetilde{\sigma}$ is also  globally defined and $S=\widetilde{\sigma}\circ \pi.$ 

On the other hand, if $\widetilde{\sigma}':Q\to T^*Q/\Lambda$ is another Lagrangian section and 
$$\widehat{F}(x)=\widehat{\mu}(\widetilde{\sigma}'(\pi(x)),x), \mbox{ for all } x\in M,$$
then, using the free character of the action $\widehat{\mu},$ we conclude that $\widetilde{\sigma}=\widetilde{\sigma}'.$

\end{proof}

Now, suppose that for a fibration $(M,Q,\pi)$, we have  actions $\phi: G \times Q \to Q$ and $\Phi: G \times M \to M$ of a Lie group $G$ on $Q$ and $M$ respectively, such that $\pi$ is $G$-equivariant, i.e. the following diagram commutes
\[
\xymatrix{M \ar[d]_\pi \ar[rr]^{\Phi_g} &&M \ar[d]^\pi \\
Q \ar[rr]^{\phi_g} &&Q
}
\]
for all $g \in G$. In such a case we say that $(M,Q,\pi,\phi,\Phi)$ is a $G$-fibration. 

\begin{definition}
Let $(M,Q,\pi,\w)$ be a Labrangian fibration and $\Phi,\phi$ be actions on $M$ and $Q$ respectively,  such that $(M,Q,\pi,\phi,\Phi)$ is a $G$-fibration. Then $(M,Q,\pi,\w,\phi,\Phi)$
 is a {$G$-Lagrangian fibration} if the action $\Phi$ is symplectic.   
\end{definition}

A first example of $G$-Lagrangian fibration is $(T^*Q,Q,\piq,\w_Q,\phi,\lc)$ when we consider an action $\phi: G \times Q \to Q$ of a Lie group $G$ on the manifold $Q$. 

\medskip

Now, fix $(M,Q,\pi,\w,\phi,\Phi)$ a complete $G$-Lagrangian fibration. We can consider the action $\mu$ given by \eqref{mu}. Since $\Phi_g$ is symplectic, then, for all $\alpha\in \Omega^1(Q),$ 
$$i_{(\Phi_g)_*X_{\pi^*\alpha}}\w=(\Phi^*_{g^{-1}})(i_{X_{\pi^*\alpha}}\w)=(\Phi^*_{g^{-1}})(\pi^*\alpha)=\pi^*(\phi_{g^{-1}}^*\alpha),$$
where $((\Phi_g)_*X_{\pi^*\alpha})(x)=T_{\Phi_{g^{-1}}(x)}\Phi_g ( X_{\pi^*\alpha}(\Phi_{g^{-1}}(x)))$, for every $x\in M.$ Therefore, 
$$(T_x\Phi_g)(X_{\pi^*\alpha}(x))=X_{\pi^*(\phi^*_{g^{-1}}\alpha)}(\Phi_g(x)),$$
that is, $X_{\pi^*\alpha}(x)=(T_{\Phi_g(x)}\Phi_{g^{-1}})(X_{\pi^*(\phi^*_{g^{-1}}(\alpha))}(\Phi_g(x)))$. Hence,  we obtain the $G$-equivariance of the action $\mu$, i.e.
\begin{equation}\label{mug}
\mu_{\phi_g(q)}((T^*\phi)_g\alpha_q,\Phi_g(x))=\Phi_g(\mu_q(\alpha_q,x))\;\;\; \mbox{ for all } \alpha_q\in T_q^*Q \mbox{ and } x\in \pi^{-1}(q). 
\end{equation}
An important consequence of this fact is that the Lagrangian submanifold $\Lambda$ is $G$-invariant when we consider the cotangent lift action $\lc$ on $T^*Q$. Therefore, the cotangent lifted action $\lc$ induces a $G$-action $\widetilde{T^*\phi}$ on $T^*Q/\Lambda$ and $(T^*Q/\Lambda,Q,\widetilde{\pi}_Q,\widetilde{\w}_Q,\phi,\widetilde{T^* \phi})$ is a $G$-Lagrangian fibration.

\medskip

Let $\widetilde{\Phi}:G\times T^*Q/\Lambda\to T^*Q/\Lambda$ be another action on $T^*Q/\Lambda$ such that $(T^*Q/\Lambda,Q,\widetilde{\pi}_Q,\widetilde{\w}_Q,\phi,\widetilde{\Phi})$ is a $G$-Lagrangian fibration. From (\ref{mug}) for the corresponding action $\widetilde{\mu}_Q$, we have that 
\begin{equation}\label{+}
(\widetilde{T^*\phi})_g([\alpha_q]) + \widetilde{\Phi}_g([\beta_g])=\widetilde{\Phi}_g([\alpha_q + \beta_q])
\end{equation}
for all $\alpha_q,\beta_q\in T_q^*Q.$ 

\medskip 

Then, we define the section $\Sigma(g): Q\to T^*Q/\Lambda$ of $\widetilde{\pi}_Q:T^*Q/\Lambda\to Q$ as follows

$$\Sigma(g)(q)=(\widetilde{T^*\phi})_{g^{-1}}(\widetilde\Phi_g([0_q])).$$

From (\ref{+}), it follows that 

$$\begin{array}{rcl}
\Sigma(g)(q)&=&
(\widetilde{T^*\phi})_{g^{-1}}(\widetilde{\Phi}_g([\gamma_q])-(\widetilde{T^*\phi})_g([\gamma_q])),\\&=&(\widetilde{T^*\phi})_{g^{-1}}(\widetilde{\Phi}_g([\gamma_q]))-[\gamma_q],
\end{array}$$
for $\gamma_q\in T_q^*Q,$ and thus, we obtain that 

$$\widetilde{\Phi}_g([\gamma_q])= (\widetilde{T^*\phi})_{g}([\gamma_q]) + (\widetilde{T^*\phi})_{g}(\Sigma(g)(q)).$$
Therefore, since $\widetilde\Phi$ and $\widetilde{T^*\phi}$ are actions, we deduce that 
\begin{equation}\label{actionsigma}
\Sigma({gh})=\Sigma(h)+(\widetilde{T^*\phi})_{h^{-1}}\circ \Sigma(g)\circ \phi_h,
\end{equation}
for all $g,h\in G.$ This condition means that $\Sigma:G\to \Sect(T^*Q/\Lambda)$ is a one-cocycle in the cohomology associated with the following complex: 

\begin{itemize}
\item A $n$-cochain is a map $\Sigma: G \times \stackrel{n}{\dots} \times G \to \Sect(T^*Q/\Lambda)$. We denote by $C^n(G,\Sect(T^*Q/\Lambda))$  the set of the $n$-cochains. The set of $0$-cochains is $\Sect(T^*Q/\Lambda).$
\item The coboundary operator $\delta_\phi: C^n(G,\Sect(T^*Q/\Lambda)) \to C^{n+1}(G,\Sect(T^*Q/\Lambda))$ is given by
\begin{eqnarray*}
\delta_\phi \Sigma (g_1, \dots, g_{n+1}) &=& (-1)^{n+1} \Sigma(g_2, \dots, g_{n+1}) + \\
&&+ \sum^n_{i=1} (-1)^{n+i+1} \Sigma(g_1, \dots, g_{i-1}, g_ig_{i+1}, \dots, g_{n+1}) +\\
&&+(\widetilde{T^*\phi})_{g_{n+1}^{-1}} \circ \Sigma(g_1, \dots, g_n)\circ \phi_{g_{n+1}}. 
\end{eqnarray*}
\end{itemize}

Denote by $\Sect_L(T^*Q/\Lambda)$ the subspace of Lagrangian sections in the Lagrangian fibration  $\widetilde\pi_Q:(T^*Q/\Lambda,\widetilde\w_Q)\to Q$.  Since $\widetilde{T^*\phi}$  is a symplectic action  for $(T^*Q/\Lambda, \widetilde{\w}_Q)$, then   $C^\bullet(G,\Sect_L(T^*Q/\Lambda))$ determines a subcomplex of $(C^\bullet(G,\Sect(T^*Q/\Lambda)), \delta_\phi).$ 

\medskip

In  general, we have the following result. 
\begin{proposition}\label{FL1}
Let $(M,Q,\pi,\w)$ be a complete Lagrangian fibration and $\phi:G\times Q\to Q$ be an action on $Q$. 
\begin{enumerate}
\item If $\Phi^1,\Phi^2$ are two symplectic lifts of $\phi:G\times Q\to Q$  then, for all $g\in G,$  there exists  a Lagrangian section $\Sigma(g): Q\to T^*Q/\Lambda$ of $\widetilde{\pi}_Q:T^*Q/\Lambda\to Q$ such that 
\begin{equation}\label{P1P2}
\Phi^2_g(x)=\Phi^1_g(\widehat{\mu}(\Sigma(g)({\pi(x)}), x)),  \mbox{ for $x\in M,$ }
\end{equation}
where $\widehat{\mu}:T^*Q/\Lambda \times M \to M$ is the free fibered action given in  (\ref{muuu}). 
 \item If $(M,Q,\pi,\w,\phi,\Phi)$ is a $G$-Lagrangian fibration, for $\Sigma:G\to \Sect(T^*Q/\Lambda)$,  the map $\Phi^\Sigma:G\times M\to M$ given by $$\Phi_g^\Sigma(x)=\Phi_g(\widehat{\mu}(\Sigma(g)({\pi(x)}),x)),$$ for $x\in M,$ defines an action of $G$ on $M$ if and only if the map $\Sigma$ is a one-cocycle in the cohomology complex $(C^\bullet(G,\Sect(T^*Q/\Lambda)), \delta_\phi).$ 
\item In the hypothesis of $(ii)$, the action $\Phi^\Sigma$ is also symplectic  if and only if the map $\Sigma$ is a one-cocycle in the cohomology subcomplex $(C^\bullet(G,\Sect_L(T^*Q/\Lambda)), \delta_\phi).$ 
\end{enumerate}
 \end{proposition}

\begin{proof}
(i) Consider, for all $g\in G,$ the sympectomorphism  $\widehat{F}_g: M\to M$ given by $\widehat{F}_g=\Phi^1_{g^{-1}}\circ \Phi^2_{g}$. Then, if we apply  Proposition \ref{lemma2} to $\widehat{F}_g,$  we conclude that (\ref{P1P2}) holds.

\medskip

(ii) The condition $\Phi^\Sigma$ is an action is equivalent with the relation 
$$\Phi_h(\widehat{\mu}(\Sigma({gh})({\pi(x)}),x))=\widehat{\mu}(\Sigma(g)({\phi_h(\pi(x))}),\Phi_h(\widehat\mu(\Sigma(h)(\pi(x)),x))).$$
Now, from (\ref{mug}), we have that 
$$\begin{array}{lcr}
\widehat\mu((\widetilde{T^*\phi})_h(\Sigma(gh)(\pi(x))),\Phi_h(x))&=&\widehat\mu(\Sigma(g)(\phi_h(\pi(x))), \widehat\mu((\widetilde{T^*\phi})_h(\Sigma(h)(\pi(x))),\Phi_h(x)))\\
&=& \widehat\mu(\Sigma(g)(\phi_h(\pi(x))) + (\widetilde{T^*\phi})_h(\Sigma(h)(\pi(x))), \Phi_h(x)).
\end{array}$$
Therefore, the free character of $\widehat\mu,$ implies that 
$$\Sigma({gh})=\Sigma(h)+(\widetilde{T^*\phi})_{h^{-1}}\circ \Sigma(g)\circ \phi_h,$$
i.e. $\Sigma$ is a one-cocyle for $(C^\bullet(G,\Sect(T^*Q/\Lambda), \delta_\phi).$ 

(iii) Let $\widehat{G}_g$ be the diffeomorphism given by 
$$\widehat{G}_g(x)=\widehat\mu(\Sigma(\pi(x)),x),\mbox{ for }x\in M.$$

Then, it is clear that 
$$\Phi_g^\Sigma=\Phi_g\circ \widehat{G}_g.$$

This, using Proposition \ref{lemma2},  implies that $\Sigma(g)$ is a Lagrangian section for the fibration $\widetilde{\pi}_Q:T^*Q/\Lambda\to Q.$ 
\end{proof}

\medskip

Moreover, the cohomology $H^\bullet(G,\phi, \Sect_L(T^*Q/\Lambda))$ of the complex $(C^\bullet((G,\Sect_L(T^*Q/\Lambda)), \delta_\phi)$ classifies the symplectic  actions on a  complete Lagrangian fibration on $Q$ whose projection is a fixed action $\phi$ on $Q$.

\begin{theorem}\label{FL2}
Let $\phi: G \times Q \to Q$ be an action of the Lie group $G$ on a manifold $Q$ and $(M,Q,\pi,\w)$ be a complete Lagrangian fibration on $Q$. 
Consider $\Sigma^1, \Sigma^2: G \to \Sect_L{(T^*Q/\Lambda)}$  two  one-cocycles in the cohomology complex  $(C^\bullet(G,\Sect_L(T^*Q/\Lambda)),\delta_\phi)$, and $\Phi^{\Sigma^1}$ and $\Phi^{\Sigma^2}$ their corresponding symplectic actions of $G$ on $M$. Then, there  exists a symplectomorphism $\widehat F:M\to M$  such that 
\begin{eqnarray}
\label{F fibradatilde}          &&\pi \circ \widehat F= \pi \mbox{ and} \\ 
\label{F equivariantetilde} && \widehat F \circ  \Phi^{\Sigma^1}_g=\Phi^{\Sigma^2}_g \circ \widehat F, \mbox{ for all } g \in G 
\end{eqnarray}
if and only if $[\Sigma^1]=[\Sigma^2] \in H^1(G,\phi, \Sect_L(T^*Q/\Lambda))$.  
\end{theorem}
\begin{proof}
Suppose $\widehat F:M\to M$ is a symplectomorphism and that  \eqref{F fibradatilde} and \eqref{F equivariantetilde} hold. Then, the conditions of Proposition \ref{lemma2} work  and therefore, there exists a Lagrangian section  $\widetilde{\sigma}:Q\to T^*Q/\Lambda$ of $\widetilde\pi_Q$ such that $\widehat F(x)=\widehat \mu(\widetilde{\sigma}(\pi(x)),x)$. Since $\widehat F \circ \Phi^{\Sigma^1}_g=\Phi^{\Sigma^2}_g \circ \widehat F$, we have that
$$\widehat \mu(\widetilde{\sigma}(\phi_g(\pi(x))), \Phi_g(\widehat{\mu}(\Sigma^1(g)({\pi(x)}),x)))=\Phi_g(\widehat \mu(\Sigma^2(g)({\pi(x)}), \widehat{\mu}(\widetilde{\sigma}(\pi(x)), x))$$
for all $x\in M.$

Thus, using (\ref{mug}), it follows that 

$$
\widehat{\mu}(\widetilde{\sigma}(\phi_g(\pi(x))),\widehat{\mu}((\widetilde{T^*\phi})_g(\Sigma^1(g)(\pi(x))),\Phi_g(x))=\widehat\mu((\widetilde{T^*\phi})_g(\Sigma^2(g)(\pi(x)) + \widetilde\sigma(\pi(x)),\Phi_g(x))),$$
which implies that 
$$
\widehat{\mu}(\widetilde{\sigma}(\phi_g(\pi(x))) + (\widetilde{T^*\phi})_g(\Sigma^1(g)(\pi(x))),\Phi_g(x))=\widehat\mu((\widetilde{T^*\phi})_g(\Sigma^2(g)(\pi(x)) + (\widetilde{T^*\phi})_g(\widetilde{\sigma}(\pi(x))),\Phi_g(x)).$$

Therefore, since the action $\widehat\mu$ is free, we deduce that

$$\Sigma^2(g)-\Sigma^1(g)=(\widetilde{T^*\phi})_{g^{-1}}\circ \widetilde{\sigma}\circ \phi_g-\widetilde{\sigma}=\delta_\phi(-\widetilde{\sigma})(g),$$
that is, $[\Sigma^1]=[\Sigma^2]$. 

\medskip

Conversely, if $[\Sigma^1]=[\Sigma^2]$, then there exists a Lagrangian section  $\widetilde{\sigma}:Q\to T^*Q/\Lambda$ of $\widetilde\pi_Q$ such that $$\Sigma^1(g)-\Sigma^2(g)=\delta_\phi(\widetilde{\sigma})(g)=\widetilde{\sigma}-(\widetilde{T^*\phi})_{g^{-1}}\circ \widetilde{\sigma}\circ \phi_g.$$ Thus, the map $\widehat{F}:M\to M$ given by 
$$\widehat{F}(x)=\widehat \mu(\widetilde{\sigma}(\pi(x)),x), \mbox{ for } x\in M,$$
satisfies the conditions of the theorem. Note that 
$$\widehat{F}^{-1}(y)=\widehat \mu(-\widetilde{\sigma}(\pi(y)),y), \mbox{ with } y\in M.$$
\end{proof}

In the case of the symplectic reference associated to a complete $G$-Lagrangian fibration $(M,Q,\pi,\w,$ $\phi,\Phi)$ we have 

\begin{corollary}
Let $(T^*Q/\Lambda,Q,\widetilde{\pi}_Q,\widetilde{\w}_Q,\phi,\widetilde{T^*\phi})$  be the $G$-symplectic reference  of a complete $G$-Lagrangian fibration  $(M,Q,\pi,\w,\phi,\Phi)$. Then: 
\begin{enumerate}
\item Every  symplectic action $\Phi:G\times T^*Q/\Lambda\to T^*Q/\Lambda$ which  projects on  $\phi$ is given by 
$$\Phi_g=\widetilde{(T^*\phi)}_g\circ t_{\Sigma(g)}$$
where $\Sigma(g):Q\to T^*Q/\Lambda$ is a section of $\widetilde{\pi}_Q:T^*Q/\Lambda\to Q$  and $t_{\Sigma(g)}:T^*Q/\Lambda\to T^*Q/\Lambda$ is the translation $t_{\Sigma_g}([\gamma_q])=[\gamma_q] + \Sigma(g)(q)$. 
\item If $\Sigma: G\to \Sect(T^*Q/\Lambda)$ is a map then $\widetilde\Phi_g^\Sigma=(\widetilde{T^*\phi})_g\circ t_{\Sigma(g)}$ defines an  action on $T^*Q/\Lambda$  if and only if the map $\Sigma$ is a one-cocycle in the cohomology complex $(C^\bullet((G,\Sect(T^*Q/\Lambda)), \delta_\phi).$
\item 
$\widetilde\Phi_g^\Sigma$ defines a  symplectic action  if and only if the map $\Sigma$ is a one-cocycle in the cohomology complex $(C^\bullet((G,\Sect_L(T^*Q/\Lambda)), \delta_\phi).$ 
\item 
Consider $\Sigma^1, \Sigma^2: G \to \Sect_L{(T^*Q/\Lambda)}$ two one-cocycles in the cohomology complex $(C^\bullet(G,\Sect_L(T^*Q/\Lambda)),$ $\delta_\phi)$, and $\widetilde\Phi^{\Sigma^1}$ and $\widetilde\Phi^{\Sigma^2}$ their corresponding symplectic actions on $(T^*Q/\Lambda, \widetilde\w_Q).$ Then, there exists a symplectomorphism $\widetilde F:T^*Q/\Lambda\to T^*Q/\Lambda$  such that 
$$\begin{array}{rcl}       
&& \widetilde\pi_Q \circ \widetilde F= \widetilde\pi_Q \mbox{ and } \\ 
&& \widetilde F \circ  \tilde\Phi^{\Sigma^1}_g=\tilde\Phi^{\Sigma^2}_g \circ \widetilde F, \mbox{ for all } g \in G 
\end{array}
$$
if and only if $[\Sigma^1]=[\Sigma^2] \in H^1(G,\phi, \Sect_L(T^*Q/\Lambda))$.  
\end{enumerate}
\end{corollary}

Finally, another application of our results is related with magnetic cotangent bundles. 

Indeed, suppose that $\beta$ is a closed $2$-form  on $Q$ and consider the symplectic structure on $T^*Q$ given by $\w_{Q,\beta}:=\w_Q+\piq^*\beta$.  Then, it is easy to prove that 
 $\pi_Q:(T^*Q, \w_{Q,\beta})\to Q$ is a complete Lagrangian fibration. Moreover, in this case, the Lagrangian submanifold $\Lambda$ of $T^*Q$ is just the zero section and, thus, the symplectic reference of  $\pi_Q:(T^*Q, \w_{Q,\beta})\to Q$ is the standard canonical projection  $\pi_Q:(T^*Q, \w_{Q})\to Q.$
 
 On the other hand, if $\phi:G\times Q\to Q$ is an action of $G$ on $Q$  and $g\in G$ then 
$$(T^*\phi)^*_g(\w_{Q,\beta})=\omega_Q + \pi_Q^*(\phi_g^*\beta).$$

So, if $\beta$ is $G$-invariant, we deduce that the cotangent lift of $\phi$ is a symplectic action for the total space of the complete Lagrangian fibration 
$$\pi_{Q,\beta}:(T^*Q,\w_{Q,\beta}) \to Q.$$
Using the previous facts,  Proposition \ref{FL1} and Theorem \ref{FL2}, we deduce the following result. 

\begin{corollary}
Let $\phi:G\times Q\to  Q$ be an action of the Lie group $G$ on the manifold $Q$ and $\beta$ be a closed $2$-form which is $G$-invariant. 

\begin{enumerate}
\item If $\Phi:G\times T^*Q\to T^*Q $ is a symplectic  action of a Lie group $G$ on  $(T^*Q,\w_Q+\pi_Q^*\beta)$ whose projection on $Q$ is the action $\phi:G\times Q\to Q,$  then
there exists a differentiable map  $A:G \times Q \to T^*Q$ such that  
\begin{enumerate}
\item$A$ is fibered on $Q$, i.e. $\pi_Q\circ A=pr_2, $
\item For each $g\in G$, the $1$-form $A_g$ on $Q$ is closed.  
\item The action $\Phi$ is given by 
$$
\Phi_g=(\lc)_g \circ t_{A(g)}, \mbox{ for all } g \in G.
$$
\end{enumerate} 
\item If $A:G\times Q\to T^*Q$ is a smooth map and $\Phi^A: G \times T^*Q \to T^*Q$ is given by $\Phi_g^A=(\lc)_g \circ t_{A(g)}$ then $\Phi^A$  is an action if and only if $A$ is a one-cocycle in the cohomology complex $(C^\bullet(G,\Omega^1(Q)),$ $\delta_\phi)$. 
\item $\Phi^A$  is also symplectic if and only if $A(g)$ is a closed 1-form on $Q,$ for all $g \in G$.
\item
If $A, B: G \to \Omega^1_c(Q)$ are  two one-cocycles in the cohomology complex $(C^\bullet(G,\Omega_c^1(Q)),\delta_\phi)$, and $\Phi^A$ and $\Phi^B$ their respective affine symplectic actions on $(T^*Q, \w_Q + \pi^*\beta)$, then  exists a symplectomorphism $F:T^*Q\to T^*Q$  such that 
$$\begin{array}{rcl}       
&& \piq \circ F= \piq \mbox{ and }\\ 
&& F \circ \Phi^A_g=\Phi^B_g \circ F, \mbox{ for all } g \in G 
\end{array}$$
if and only if $[A]=[B] \in H^1(G,\phi, \Omega^1_c(Q))$.  
\end{enumerate}

\end{corollary}

\end{document}